\documentclass[12pt]{amsart}

\makeatletter
\def\blfootnote{\gdef\@thefnmark{}\@footnotetext}
\makeatother

\usepackage{epsfig}
\usepackage{graphics}
\usepackage{dcpic, pictexwd}

\usepackage[colorlinks=true,
                    linkcolor=blue,
                    urlcolor=blue,
                    citecolor=blue,
                    anchorcolor=blue]{hyperref}
\usepackage{mathtools}
\usepackage{amsmath,amssymb}

\theoremstyle{plain}

\newtheorem*{theorem*}{Theorem}
\newtheorem*{thma}{Theorem A}
\newtheorem*{thmb}{Theorem B}

\newtheorem{theorem}{Theorem}[section]

\newtheorem{lemma}[theorem]{Lemma}

\theoremstyle{Acknowledgments}

\theoremstyle{definition}

\theoremstyle{remark}
\newtheorem*{remark}{Remark}

\def\mod{{\rm Mod}}

\begin{document}
\blfootnote{\textup{2000} \textit{Mathematics Subject Classification}:
57M07, 20F05, 20F38}
\newenvironment{prooff}{\medskip \par \noindent {\it Proof}\ }{\hfill
$\square$ \medskip \par}
    \def\sqr#1#2{{\vcenter{\hrule height.#2pt
        \hbox{\vrule width.#2pt height#1pt \kern#1pt
            \vrule width.#2pt}\hrule height.#2pt}}}
    \def\square{\mathchoice\sqr67\sqr67\sqr{2.1}6\sqr{1.5}6}
\def\pf#1{\medskip \par \noindent {\it #1.}\ }
\def\endpf{\hfill $\square$ \medskip \par}
\def\demo#1{\medskip \par \noindent {\it #1.}\ }
\def\enddemo{\medskip \par}
\def\qed{~\hfill$\square$}

 \title[Generating $\mod^{*}(\Sigma_{g,p})$ by Three Involutions] {Generating the Extended Mapping Class Group by Three Involutions}

\author[T{\"{u}}l\.{i}n Altun{\"{o}}z,       Mehmetc\.{i}k Pamuk, and O\u{g}uz Y{\i}ld{\i}z ]{T{\"{u}}l\.{i}n Altun{\"{o}}z,    Mehmetc\.{i}k Pamuk, and Oguz Yildiz}

\address{Department of Mathematics, Middle East Technical University,
 Ankara, Turkey}
\email{atulin@metu.edu.tr}  \email{mpamuk@metu.edu.tr} \email{oguzyildiz16@gmail.com}

\begin{abstract}
We prove that the extended mapping class group, $\mod^{*}(\Sigma_{g})$, of a  connected orientable  surface of genus $g$,  can be generated by three involutions
for $g\geq 5$.  In the presence of punctures, we prove that $\mod^{*}(\Sigma_{g,p})$ can be generated by three involutions for $g\geq 10$ and $p\geq 6$ (with the exception that for $g\geq 11$, $p$ should be at least $15$).
\end{abstract}
\maketitle
  \setcounter{secnumdepth}{2}
 \setcounter{section}{0}
 
\section{Introduction}
Let $\Sigma_{g,p}$ denote a connected orientable surface of genus $g$ with $p\geq0$ punctures. When $p=0$, we drop it from the notation and write $\Sigma_{g}$. 
The mapping class group of $\Sigma_g$ is the group of isotopy classes of orientation preserving diffeomorphisms and is denoted by $\mod(\Sigma_g)$. 
It is a classical result that $\mod(\Sigma_g)$ is generated by finitely many Dehn twists about nonseparating simple closed curves~\cite{de,H,l3}.  
The study of algebraic properties of mapping class group, finding small generating sets, generating sets with particular properties, is an active one leading to interesting developments. 
Wajnryb~\cite{w} showed that $\mod(\Sigma_g)$ can be generated by two elements given as a product of Dehn twists. As the group is not abelian, this is the smallest 
possible. Korkmaz~\cite{mk2} improved this result by first showing that one of the two generators can be taken as a Dehn twist and the other as a torsion element. 
He also proved that $\mod(\Sigma_g)$ can be generated by two torsion elements. Recently, the third author showed that $\mod(\Sigma_g)$ is generated by two torsions of small orders~\cite{y1}.

Generating  $\mod(\Sigma_g)$ by involutions was first considered by McCarthy and Papadopoulus~\cite{mp}.  They showed that the group can be generated 
by infinitely many conjugates of a single involution (element of order two) for $g\geq 3$.     
In terms of generating by finitely many involutions, Luo~\cite{luo} showed that any Dehn twist about a nonseparating simple closed curve 
can be written as a product six involutions, which in turn implies that $\mod(\Sigma_g)$ can be generated by $12g+6$ involutions.  
Brendle and Farb~\cite{bf} obtained a generating set of six involutions for $g\geq3$. Following their work, Kassabov~\cite{ka} showed that 
$\mod(\Sigma_g)$ can be generated by four involutions if $g\geq7$.  Recently, Korkmaz~\cite{mk1} showed that $\mod(\Sigma_g)$ is generated by three involutions 
if $g\geq8$ and four involutions if $g\geq3$. The third author improved these results by showing that this group can be generated by three involutions if $g\geq6$~\cite{y2}.

The extended mapping class group $\mod^{*}(\Sigma_{g})$  is defined to be the group of isotopy classes of all self-diffeomorphisms of $\Sigma_{g}$. 
The mapping class group $\mod(\Sigma_{g})$ is an index two normal subgroup of $\mod^{*}(\Sigma_{g})$.
In~\cite{mk2}, it is proved that $\mod^{*}(\Sigma_{g})$ can be generated by two elements, one of which is a Dehn twist.
Moreover, it follows from~\cite[Theorem $14$]{mk2} that $\mod^{*}(\Sigma_{g})$ can be generated by three torsion elements for $g\geq1$. 
Also, Du~\cite{du1, du2} proved that $\mod^{*}(\Sigma_{g})$ can be generated by two torsion elements of order $2$ and $4g+2$ for $g\geq 3$.
In terms of involution generators,  as it contains nonabelian free groups, the minimal number of involution generators is three and 
Stukow~\cite{st} proved that $\mod^{*}(\Sigma_{g})$ can be generated by three involutions for $g\geq1$.  
Although our main interest in this paper is to find minimal generating sets for the extended mapping class group in the presence of punctures, in Section~\ref{S3}, 
we test our techniques to find minimal generating sets of involutions.  In this direction, we obtain the following result(see Theorems~\ref{thm2} and ~\ref{thm3}):  
\begin{thma}
For $g\geq 5$, the extended mapping class group $\mod^{*}(\Sigma_{g})$ can be generated by three involutions.
\end{thma}

In the presence of punctures, the mapping class group $\mod(\Sigma_{g,p})$ is defined to be the group of isotopy classes of orientation-preserving 
self-diffeomorphisms of $\Sigma_{g,p}$ preserving the set of punctures. The extended mapping class group $\mod^{*}(\Sigma_{g,p})$
is defined as the group of isotopy classes of all (including orientation-reversing) self-diffeomorphisms of $\Sigma_{g,p}$ that preserve the set of punctures.
Kassabov~\cite{ka} gave involution generators of $\mod(\Sigma_{g,p})$,  proving that this group
can be generated by four involutions if $g>7$ or $g=7$ and $p$ is even, five involutions if $g>5$ or $g=5$ and $p$ is even, six involutions if $g>3$ or $g=3$ and $p$ is even
(Allowing orientation reversing involutions these results can also be used for $\mod^{*}(\Sigma_{g,p})$ \cite[Remark~$3$]{ka}). 
Later, Monden~\cite{m1} removed the parity conditions on the number of punctures. For $g\geq1$ and $p\geq2$, he~\cite{m2} also proved that $\mod(\Sigma_{g,p})$ can be 
generated by three elements, one of which is a Dehn twist. Moreover, he gave a similar generating set for $\mod^{*}(\Sigma_{g,p})$ consisting of three elements. Recently, Monden showed that $\mod(\Sigma_{g,p})$ and $\mod^{*}(\Sigma_{g,p})$ are generated by two elements~\cite{m3}.

In Section~\ref{S4}, we prove the following result, giving a partial answer to Question $5.6$ of \cite{m1}.
\begin{thmb}\label{thmb}
For $g\geq 10$ and $p\geq 6$ (with the exception that for $g\geq 11$, $p$ should be at least $15$), the extended mapping class group $\mod^{*}(\Sigma_{g,p})$ can be generated by three involutions.
\end{thmb}

\begin{remark}
At the end of the paper, we also show that the same result holds for $g\geq 10$ and $p=1,2,3$.
\end{remark}

Before we finish the introduction, let us point out that by the version of Dehn-Nielsen-Baer theorem for punctured surfaces (see \cite[Section~8.2.7]{FM}), 
$\mod^{*}(\Sigma_{g,p})$ is isomorphic to the subgroup of the outer automorphism group $Out(\pi_1(\Sigma_{g,p}))$ consisting of elements that preserve the set of 
conjugacy classes of the simple closed curves surrounding individual punctures.  Note also that these conjugacy classes are precisely the primitive conjugacy classes 
that correspond to the parabolic elements of the group of isometries of the hyperbolic plane.

\medskip

\noindent
{ Acknowledgements.}
 We would like to thank Tara Brendle for her  helpful comments.  We also would like to thank the referee for carefully reading our manuscript, pointing out an error in an earlier version and suggesting useful ideas which improved the paper.
The first author was partially supported by the Scientific and Technological Research Council of Turkey (T\"{U}B\.{I}TAK)[grant number 117F015].


\par  
\section{Background and Results on Mapping Class Groups} \label{S2}

 Let $\Sigma_{g,p}$ be a connected orientable surface of genus $g$ with $p$ punctures specified by the set $P=\lbrace z_1,z_2,\ldots,z_p\rbrace$ of $p$ distinguished points. If $p$ is zero then we omit from the notation. {\textit{The mapping class group}} 
 $\mod(\Sigma_{g,p})$ of the surface $\Sigma_{g,p}$ is defined to be the group of the isotopy classes of orientation preserving
 diffeomorphisms $\Sigma_{g,p} \to \Sigma_{g,p}$ which fix the set $P$. {\textit{The extended mapping class group}} $\mod^{*}(\Sigma_{g,p})$ of the surface $\Sigma_{g,p}$ is defined to be the group of isotopy classes of all (including orientation-reversing) diffeomorphisms of $\Sigma_{g,p}$ which fix the set $P$. Let $\mod_{0}^{*} (\Sigma_{g,p})$ denote the subgroup of $\mod^{*}(\Sigma_{g,p})$ which consists of elements fixing the set $P$ pointwise. It is obvious that we have the exact sequence:
 \[
1\longrightarrow \mod_{0}^{*}(\Sigma_{g,p})\longrightarrow \mod^{*}(\Sigma_{g,p}) \longrightarrow S_{p}\longrightarrow 1,
\]
where $S_p$ denotes the symmetric group on the set $\lbrace1,2,\ldots,p\rbrace$ and the restriction of the isotopy class of a diffeomorphism to its action on the puncture points gives the last projection. \par
Let $\beta_{i,j}$ be an embedded arc joining two punctures $z_i$ and $z_j$ and not intersecting $\delta$ on $\Sigma_{g,p}$. Let $D_{i,j}$ be a closed regular neighbourhood of $\beta_{i,j}$ such that it is a disk with two punctures. There is a diffeomorphism $H_{i,j}: D_{i,j} \to D_{i,j}$, which interchanges the punctures such that $H_{i,j}^{2}$ is the right handed Dehn twist about $\partial D_{i,j}$ and is equal to the identity on the complement of the interior of $D_{i,j}$. Such a diffeomorphism is called \textit{the (right handed) half twist} about $\beta_{i,j}$. One can extend it to a diffeomorphism of $\mod(\Sigma_{g,p})$. Throughout the paper we do not distinguish a 
 diffeomorphism from its isotopy class. For the composition of two diffeomorphisms, we
use the functional notation; if $g$ and $h$ are two diffeomorphisms, then
the composition $gh$ means that $h$ is applied first.\\
\indent
 For a simple closed 
curve $a$ on $\Sigma_{g,p}$, following ~\cite{apy,mk1} the right-handed 
Dehn twist $t_a$ about $a$ will be denoted by the corresponding capital letter $A$.

Now, let us recall the following basic properties of Dehn twists which we use frequently in the remaining of the paper. Let $a$ and $b$ be 
simple closed curves on $\Sigma_{g,p}$ and $f\in \mod^{*}(\Sigma_{g,p})$.
\begin{itemize}
\item \textbf{Commutativity:} If $a$ and $b$ are disjoint, then $AB=BA$.
\item \textbf{Conjugation:} If $f(a)=b$, then $fAf^{-1}=B^{s}$, where $s=\pm 1$ 
depending on whether $f$ is orientation preserving or orientation reversing on a 
neighbourhood of $a$ with respect to the chosen orientation.
\end{itemize}

\section{Involution generators for $\mod^{*}(\Sigma_g)$}\label{S3}
We start with this section by embedding $\Sigma_g$ into $\mathbb{R}^{3}$ so that it is invariant under the reflections $\rho_1$ and $\rho_2$ (see Figures~\ref{GOC} and~\ref{GEC}). Here, $\rho_1$ and $\rho_2$ are the reflections in the $xz$-plane so that $R=\rho_1\rho_2$ is the rotation by $\frac{2\pi}{g}$ about the $x$-axis. Now, let us recall the following set of generators given by Korkmaz~\cite[Theorem~$5$]{mk1}.

\begin{theorem}\label{thm1}
If $g\geq3$, then the mapping class group $\mod(\Sigma_g)$ is generated by the four elements $R$, $A_1A_{2}^{-1}$, $B_1B_{2}^{-1}$, $C_1C_{2}^{-1}$.
\end{theorem}
By adding an orientation reversing self-diffeomorphism to the above generating set, one can easily see that $\mod^{*}(\Sigma_g)$ can be generated by five elements. In the following theorems, we show that one can reduce the number of generators to three and all the generators are of order two.
\begin{figure}[hbt!]
\begin{center}
\scalebox{0.22}{\includegraphics{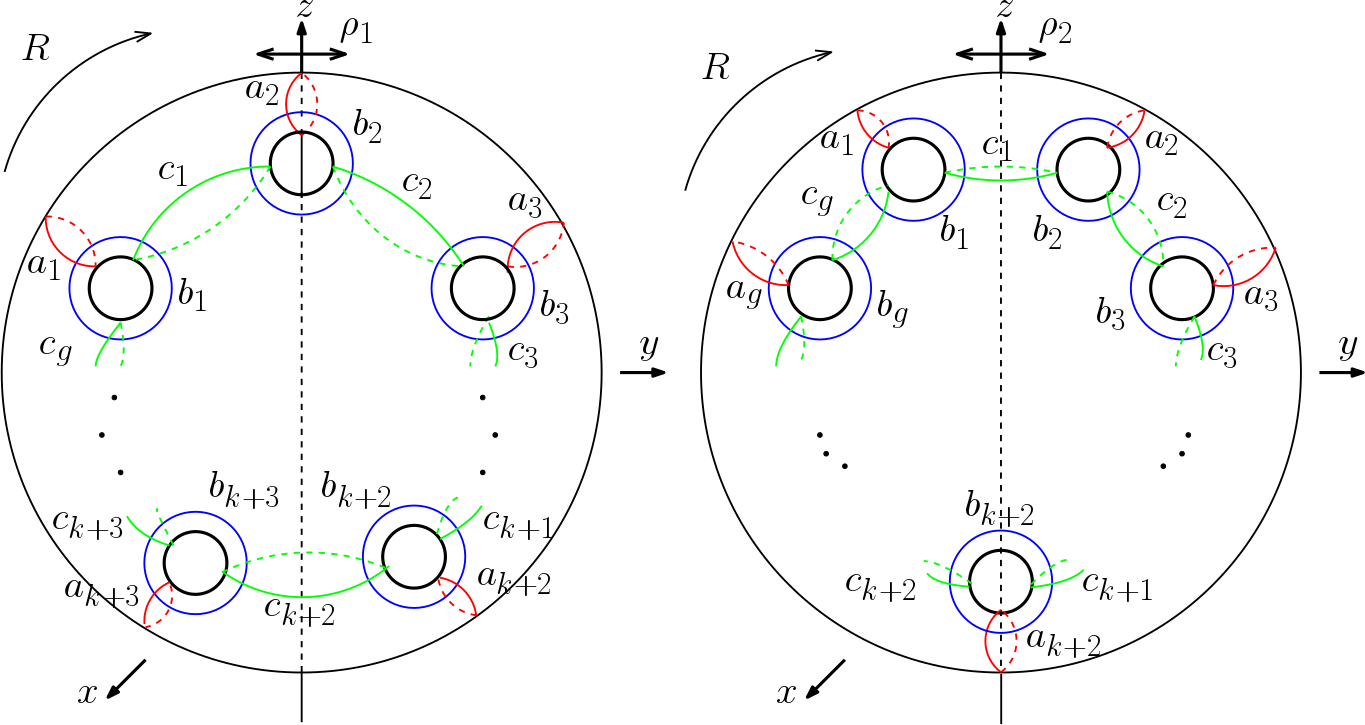}}
\caption{The reflections $\rho_1$ and $\rho_2$ on $\Sigma_{g}$ if $g=2k+1$.}
\label{GOC}
\end{center}
\end{figure}
\begin{figure}[hbt!]
\begin{center}
\scalebox{0.22}{\includegraphics{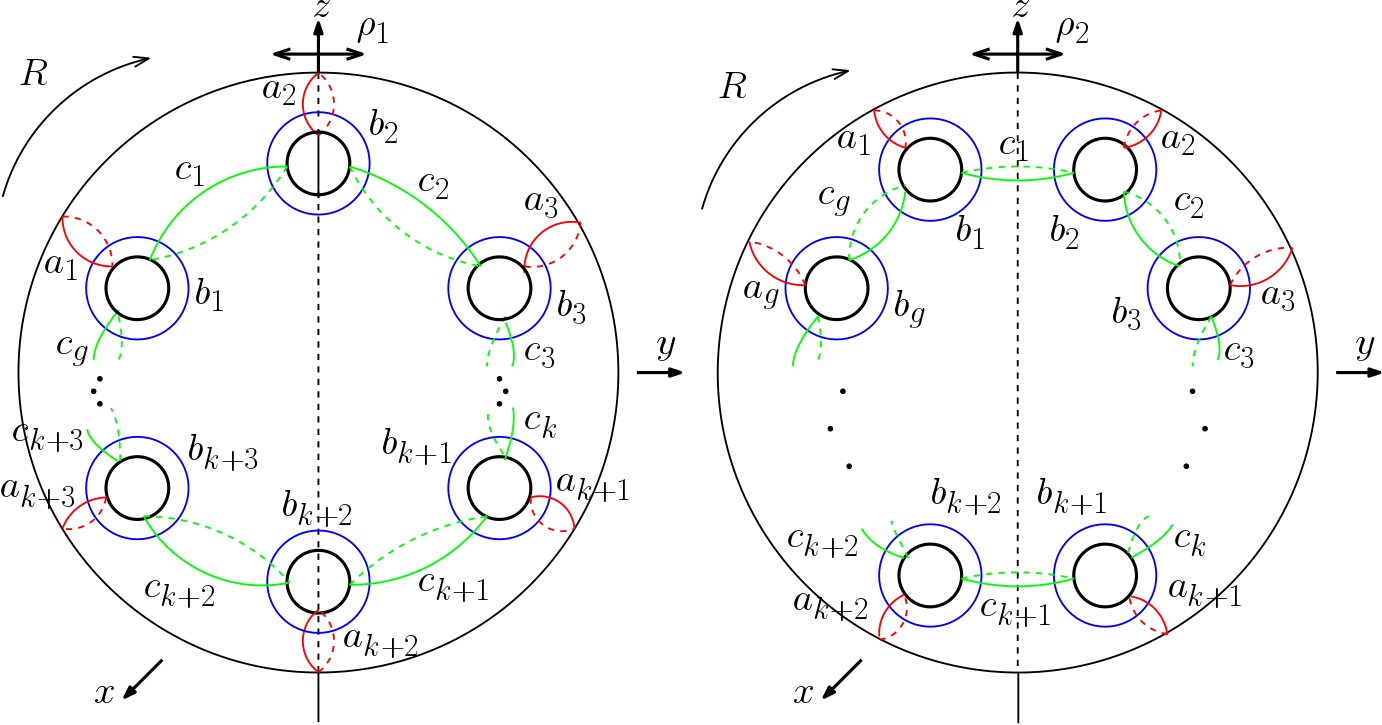}}
\caption{The reflections $\rho_1$ and $\rho_2$ on $\Sigma_{g}$ if $g=2k$.}
\label{GEC}
\end{center}
\end{figure}

\begin{theorem}\label{thm2}
If $g\geq5$ and odd, then $\mod^{*}(\Sigma_g)$ is generated by the involutions $\rho_1$, $\rho_2$ and $\rho_1A_1B_2C_{\frac{g+3}{2}} A_3$. 
\end{theorem}
\begin{proof}
Consider the surface $\Sigma_{g}$ as in Figure~\ref{GOC} and observe that the involution $\rho_1$ satisfies
\[
 \rho_1(a_1)=a_3, \rho_1(b_{2})=b_{2} \textrm{ and } \rho_1(c_{\frac{g+3}{2}})=c_{\frac{g+3}{2}}.
\]
Since $\rho_1$ reverses the orientation of a neighbourhood of any simple closed curve, we get
\[
\rho_1A_1\rho_1=A_{3}^{-1}, \rho_1B_{2}\rho_1=B_{2}^{-1} \textrm{ and } \rho_1C_{\frac{g+3}{2}}\rho_1=C_{\frac{g+3}{2}}^{-1}.
\]
It is easily seen that $\rho_1A_1B_2C_{\frac{g+3}{2}} A_3$ is an involution. Let $K$ be the subgroup of $\mod^{*}(\Sigma_{g})$ generated by the set
	\[
	\lbrace \rho_1,\rho_2, \rho_1A_1B_2C_{\frac{g+3}{2}} A_3 \rbrace.
	\]
Note that the rotation $R$ and the orientation reversing diffeomorphism $\rho_1$ (or $\rho_2$) are contained in $K$. Hence, all we need to show is that 
the elements $A_1A_{2}^{-1}, B_1B_{2}^{-1}$ and $C_1C_{2}^{-1}$ belong to $K$. For $g\geq7$ and odd, by proof of ~\cite[Theorem $3.4$]{apy}, these elements are contained in $K$. For $g=5$, the proof follows from the proof of \cite[Theorem $3.3$]{apy}. 
\end{proof}

Next, we deal with the even genera case.

\begin{theorem}\label{thm3}
If $g\geq6$ and even, then $\mod^{*}(\Sigma_g)$ is generated by the involutions $\rho_1$, $\rho_2$ and $\rho_1A_2C_{\frac{g}{2}}B_{\frac{g+4}{2}} C_{\frac{g+6}{2}}$.
\end{theorem}
\begin{proof}
Consider the surface $\Sigma_{g}$ as in Figure~\ref{GEC} when $g\geq6$ and even. The involution $\rho_1$ satisfies
\[
 \rho_1(a_2)=a_2, \rho_1(b_{\frac{g+4}{2}})=b_{\frac{g+4}{2}} \textrm{ and } \rho_1(c_{\frac{g}{2}})=c_{\frac{g+6}{2}}.
\]
Since $\rho_1$ reverses the orientation of a neighbourhood of any simple closed curve, we have
\[
\rho_1A_2\rho_1=A_{2}^{-1}, \rho_1B_{\frac{g+4}{2}}\rho_1=B_{\frac{g+4}{2}}^{-1} \textrm{ and } \rho_1C_{\frac{g}{2}}\rho_1=C_{\frac{g+6}{2}}^{-1}.
\]

 It can be shown that $\rho_1A_2C_{\frac{g}{2}}B_{\frac{g+4}{2}} C_{\frac{g+6}{2}}$ 
is an involution. Let $H$ be the subgroup of $\mod^{*}(\Sigma_{g})$ generated by the set
	\[
	\lbrace \rho_1,\rho_2, \rho_1A_2C_{\frac{g}{2}}B_{\frac{g+4}{2}} C_{\frac{g+6}{2}} \rbrace.
	\]
Note that the rotation $R$ is in $H$. Since $H$ contains the orientation reversing diffeomorphism $\rho_1$ (or $\rho_2$), again all we need to show is that 
the elements $A_1A_{2}^{-1}, B_1B_{2}^{-1}$ and $C_1C_{2}^{-1}$ are contained in $H$. By the proof of ~\cite[Theorem $3.5$]{apy}, these elements are contained in $H$.
\end{proof}

\section{Involution generators for $\mod^{*}(\Sigma_{g,p})$}\label{S4}

In this section, we introduce punctures on a genus $g$ surface and present involution generators for the extended mapping class group $\mod^{*}(\Sigma_{g,p})$. First, we recall the following basic lemma from algebra.
\begin{lemma}\label{lemma1}
Let $G$ and $K$ be groups, Suppose that 
we have the following short exact sequence holds,
\[
1 \longrightarrow N \overset{i}{\longrightarrow}G \overset{\pi}{\longrightarrow} K\longrightarrow 1.
\]
Then the subgroup $H$ contains $i(N)$ and has a surjection to $K$ if and only if $H=G$.
\end{lemma}
\par

For $G=\mod^{*}(\Sigma_{g,p})$ and $N=\mod_{0}^{*}(\Sigma_{g,p})$ (self-diffeomorphisms fixing the punctures pointwise), we have the following short exact sequence:
\[
1\longrightarrow \mod_{0}^{*}(\Sigma_{g,p})\longrightarrow \mod^{*}(\Sigma_{g,p}) \longrightarrow S_{p}\longrightarrow 1,
\]
where $S_p$ denotes the symmetric group on the set $\lbrace1,2,\ldots,p\rbrace$. 
Therefore, we have the following useful result which follows immediately from Lemma~\ref{lemma1}. Let $H$ be a subgroup of $\mod^{*}(\Sigma_{g,p})$. If the subgroup $H$ contains $\mod_{0}^{*}(\Sigma_{g,p})$ and has a surjection to $S_p$ then $H=\mod^{*}(\Sigma_{g,p})$.

\begin{figure}[hbt!]
\begin{center}
\scalebox{0.22}{\includegraphics{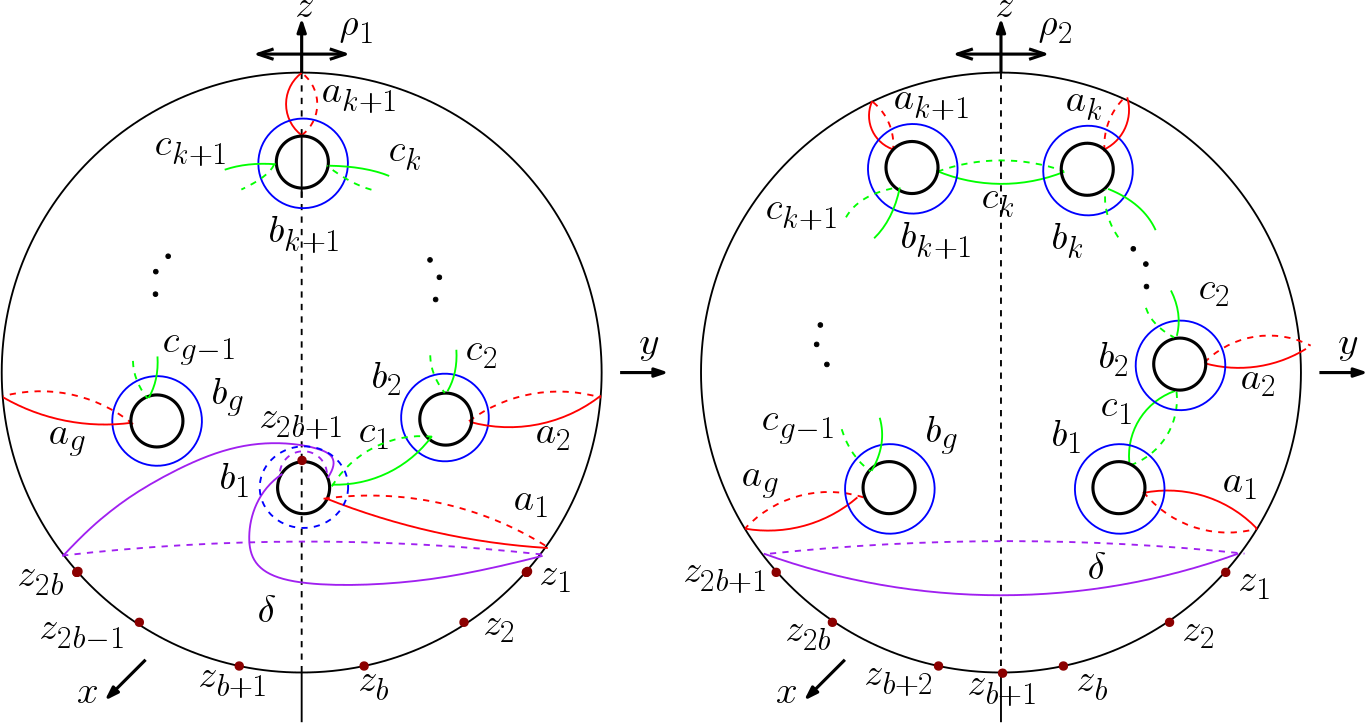}}
\caption{The reflections $\rho_1$ and $\rho_2$ if $g=2k$ and $p=2b+1$.}
\label{EO}
\end{center}
\end{figure}
\begin{figure}[hbt!]
\begin{center}
\scalebox{0.22}{\includegraphics{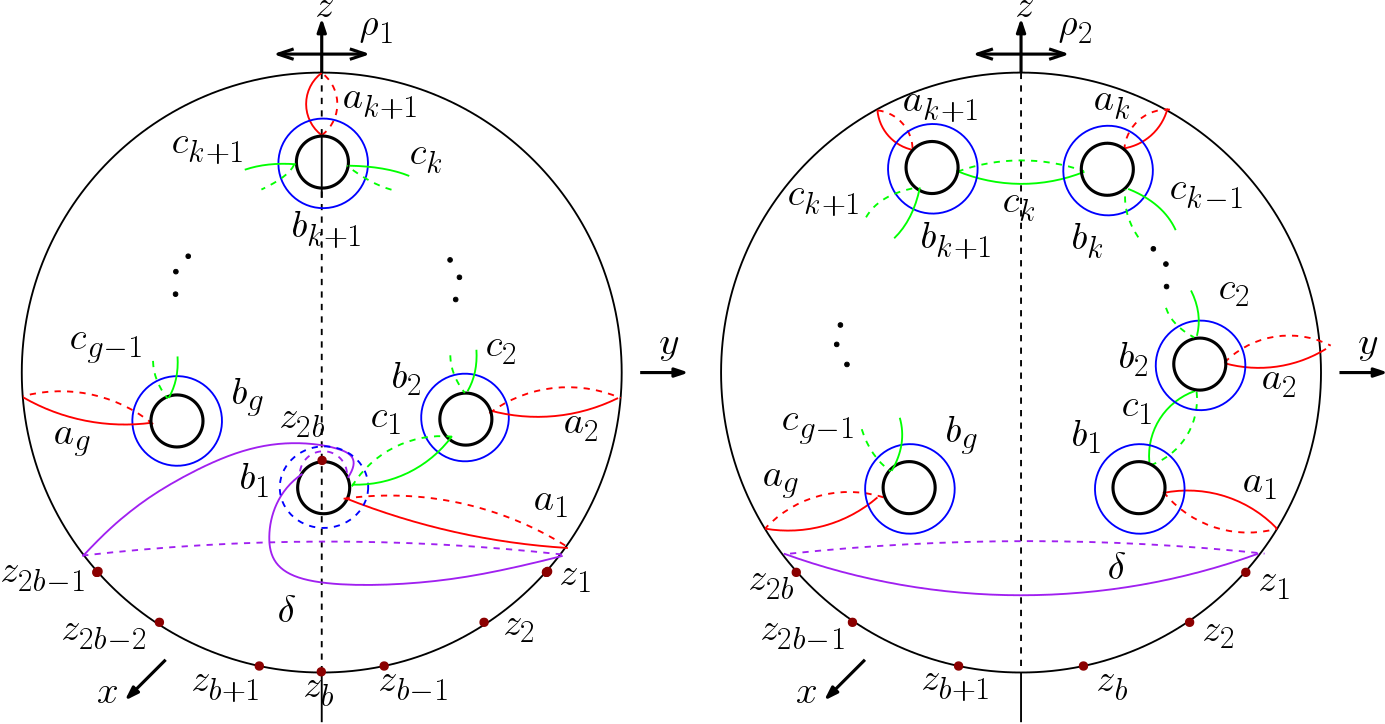}}
\caption{The reflections $\rho_1$ and $\rho_2$ if $g=2k$ and $p=2b$.}
\label{EE}
\end{center}
\end{figure}
 \begin{figure}[hbt!]
\begin{center}
\scalebox{0.2}{\includegraphics{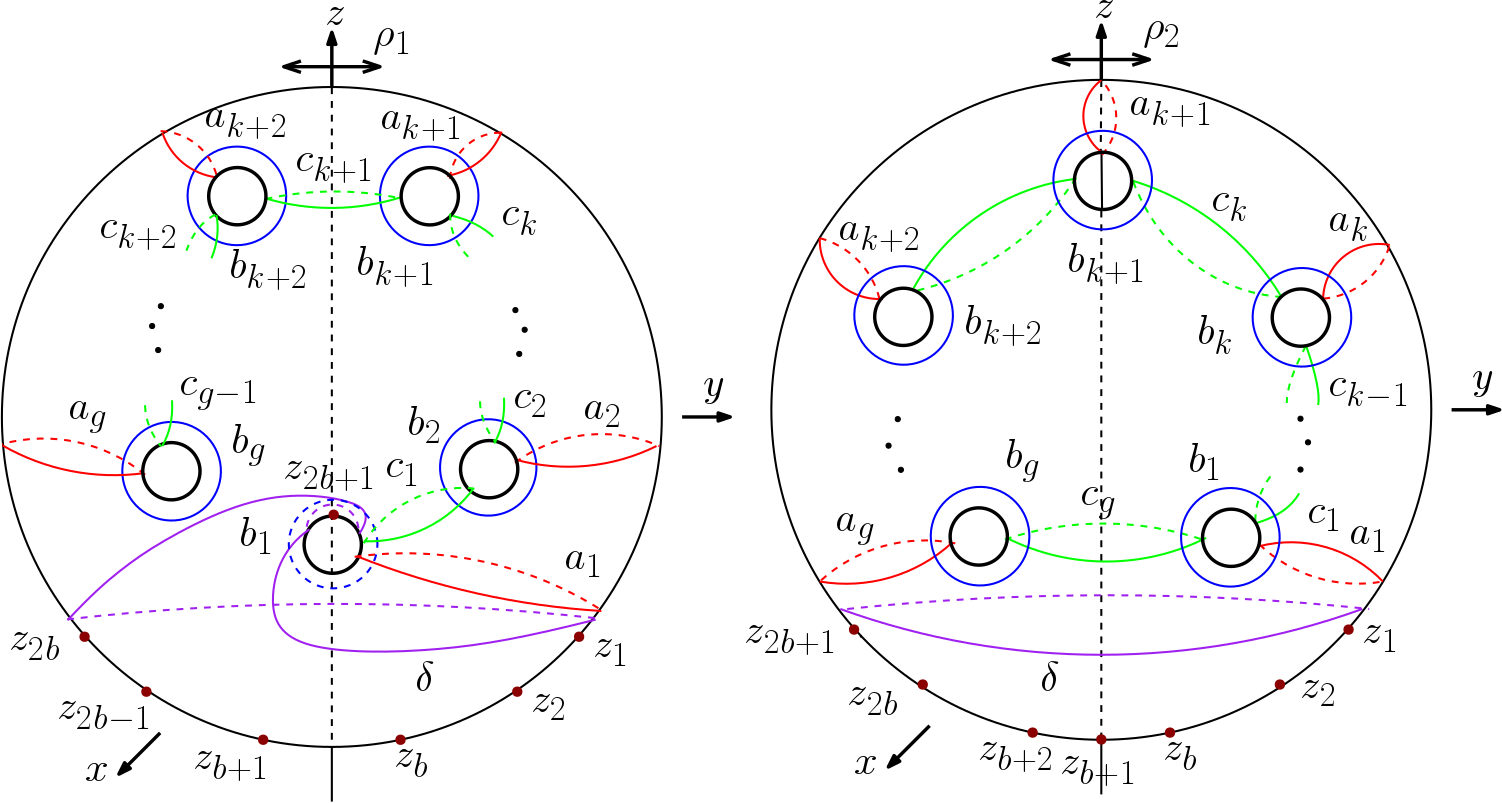}}
\caption{The reflections $\rho_1$ and $\rho_2$ if $g=2k+1$ and $p=2b+1$.}
\label{OO}
\end{center}
\end{figure}
\begin{figure}[hbt!]
\begin{center}
\scalebox{0.23}{\includegraphics{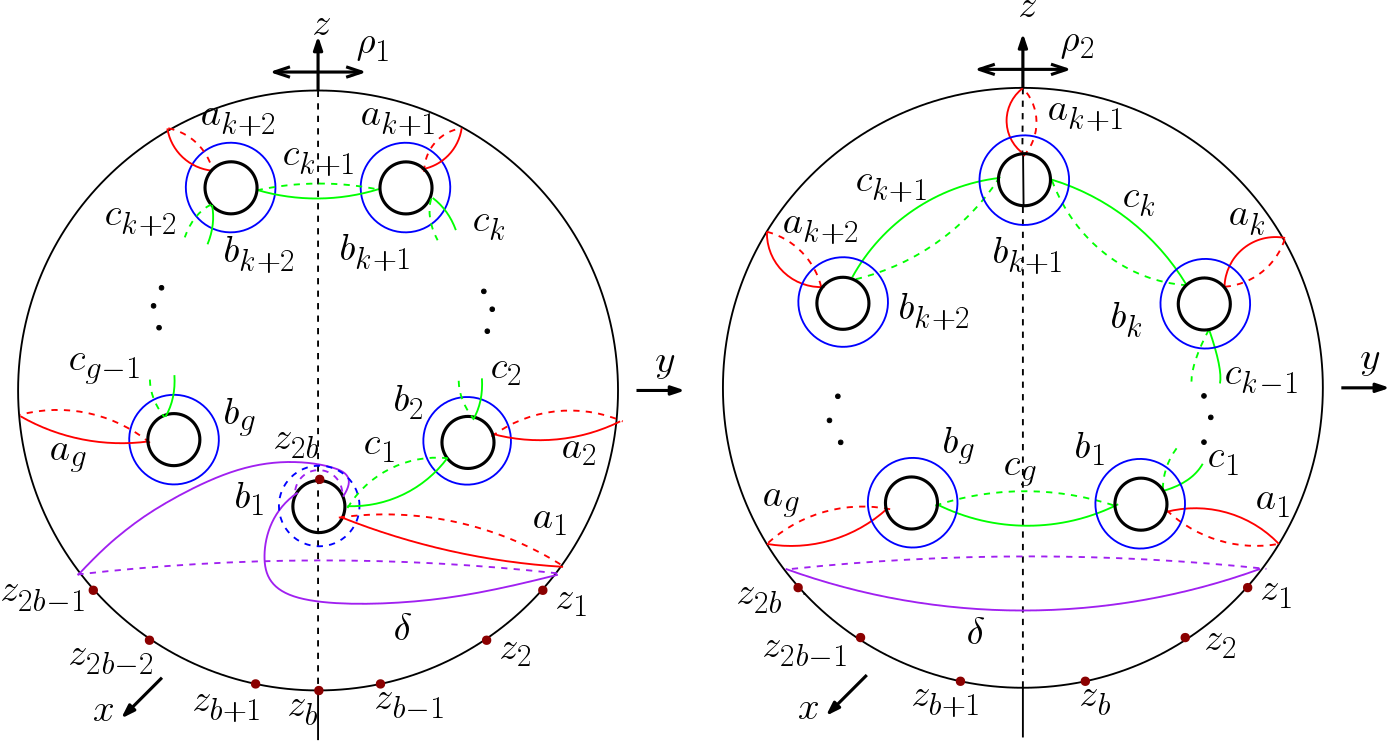}}
\caption{The reflections $\rho_1$ and $\rho_2$ if $g=2k+1$ and $p=2b$.}
\label{OE}
\end{center}
\end{figure}
In the presence of punctures, we consider the reflections $\rho_1$ and $\rho_2$ as shown in Figures~\ref{EO}, \ref{EE}, \ref{OO} and \ref{OE}. Note that the element $R=\rho_1\rho_2$ is contained in $\mod^{*}(\Sigma_{g,p})$ and we have
\begin{itemize}
\item $R(a_i)=a_{i+1}$, $R(b_i)=b_{i+1}$ for $i=1,\ldots,g-1$ and $R(b_g)=b_{1}$,
\item $R(c_i)=c_{i+1}$ for $i=1,\ldots,g-2$,
\item $R(z_1)=z_p$ and $R(z_i)=z_{i-1}$ for $i=2,\ldots,p$.
\end{itemize}
\noindent

In the proof of the following lemmata, we basically follow Theorem~\ref{thm1}.

\begin{lemma}\label{lem2k}
For $g=2k\geq 10$, the subgroup $H$  of $\mod^{*}(\Sigma_{g,p})$ generated by 
\begin{equation*}
 \begin{cases}
      \rho_1, \rho_2, \rho_2H_{b,b+2}B_{k-3}A_{k-1}C_kA_{k+2}B_{k+4} & \text {if  $p=2b+1\geq 7,$}\\
        \rho_1, \rho_2, \rho_2H_{b,b+1}B_{k-3}A_{k-1}C_kA_{k+2}B_{k+4} &  \text{if $p=2b\geq 6$}
    \end{cases}       
\end{equation*}
contains the Dehn twists $A_i$, $B_i$ and $C_j$ for $i=1,\ldots,g$ and $j=1,\ldots,g-1$.
\end{lemma}
\begin{proof}
Consider the models for $\Sigma_{g,p}$ as shown in Figures~\ref{EO} and~\ref{EE}. Start with the case $p=2b+1$. Let $E_1:=H_{b,b+2}B_{k-3}A_{k-1}C_kA_{k+2}B_{k+4}$ so that the subgroup $H$ is generated by the elements $\rho_1$, $\rho_2$ and $\rho_2E_1$. Since $H$ contains the elements $\rho_1$, $\rho_2$ and $\rho_2E_1$, it follows that $H$ also contains the elements $R=\rho_1\rho_2$ and $E_1=\rho_2\rho_2E_1$.

Let $E_2$ denote the conjugation of $E_1$ by $R^{-1}$. Since 
\[
R^{-1}(b_{k-3},a_{k-1},c_k,a_{k+2},b_{k+4})=(b_{k-4},a_{k-2},c_{k-1},a_{k+1},b_{k+3})
\]
and 
\[
R^{-1}(z_{b},z_{b+2})=(z_{b+1},z_{b+3}),
\]
it follows that $E_2=R^{-1}E_1R=H_{b+1,b+3}B_{k-4}A_{k-2}C_{k-1}A_{k+1}B_{k+3}\in H$. Let $E_3$ be the conjugation of $E_2$ by $R^3$. Since the element $R^3$ satisfies
\[
R^3
(b_{k-4},a_{k-2},c_{k-1},a_{k+1},b_{k+3})=(b_{k-1},a_{k+1},c_{k+2},a_{k+4},b_{k+6})
\]
and 
\[
R^3(z_{b+1},z_{b+3})=(z_{b-2},z_{b}),
\]
the element
\[
E_3=R^3E_2R^{-3}=H_{b-2,b}B_{k-1}A_{k+1}C_{k+2}A_{k+4}B_{k+6}\in H.
\]
Consider the element $E_4=(E_2E_3)E_2(E_2E_3)^{-1}$, which is contained in $H$. Thus,
\[
E_4=H_{b+1,b+3}B_{k-4}A_{k-2}B_{k-1}A_{k+1}C_{k+2}
\]
As we have similar cases in the remaining parts of the paper, let us explain this calculation in more details. It is easy to verify that the diffeomorphism $E_2E_3$ maps the curves $\lbrace b_{k-4},a_{k-2},c_{k-1},a_{k+1},b_{k+3} \rbrace$ to the curves $\lbrace b_{k-4},a_{k-2},b_{k-1},a_{k+1},c_{k+2}  \rbrace$, respectively. Since the half twists $H_{b+1,b+3}$ and  $H_{b-2,b}$ commute, we get
\begin{eqnarray*}
E_4&=&(E_2E_3)E_2(E_2E_3)^{-1}\\
&=&(E_2E_3)(H_{b+1,b+3}B_{k-4}A_{k-2}C_{k-1}A_{k+1}B_{k+3})(E_2E_3)^{-1}\\
&=&H_{b+1,b+3}B_{k-4}A_{k-2}B_{k-1}A_{k+1}C_{k+2}.
\end{eqnarray*}
We also get the element
\[
E_5=RE_4R^{-1}=H_{b,b+2}B_{k-3}A_{k-1}B_{k}A_{k+2}C_{k+3}\in H.
\]
Hence the subgroup $H$ contains the element 
\[E_6=E_1E_5^{-1}=C_kB_{k+4}C_{k+3}^{-1}B_{k}^{-1}.
\]
Moreover, we have the following elements:
\begin{eqnarray*}
E_7&=&RE_5R^{-1}=H_{b-1,b+1}B_{k-2}A_{k}B_{k+1}A_{k+3}C_{k+4},\\
E_8&=&R^{-3}E_7R^{3}=H_{b+2,b+4}B_{k-5}A_{k-3}B_{k-2}A_{k}C_{k+1} 
\textrm{ and }\\
E_9&=&(E_7E_8)E_7(E_7E_8)^{-1}=H_{b-1,b+1}B_{k-2}A_{k}C_{k+1}A_{k+3}C_{k+4},
\end{eqnarray*}
are contained in $H$.
Thus, we obtain the element $E_7E_{9}^{-1}=B_{k+1}C_{k+1}^{-1} \in H$. By conjugating $B_{k+1}C_{k+1}^{-1}$ with powers of $R$, we have $B_{i}C_{i}^{-1} \in H$ for all $i=1,\ldots,g-1$. Moreover, the element $
E_6(B_kC_k^{-1})=B_{k+4}C_{k+3}^{-1}
$ is contained in $H$. Thus, each $B_{i+1}C_i^{-1}$ is in $H$ for all $i=1,\ldots,g-1$ by conjugating this element with powers of $R$. Consider the elements
 \begin{eqnarray*}
 E_{10}&=&(B_{k}C_{k}^{-1})(B_{k+5}C_{k+4}^{-1})(C_{k+4}B_{k+4}^{-1})E_1\\
 &=&H_{b,b+2}B_{k-3}A_{k-1}B_kA_{k+2}B_{k+5}, \\
 E_{11}&=&R^{-1}E_{10}R=H_{b+1,b+3}B_{k-4}A_{k-2}B_{k-1}A_{k+1}B_{k+4}\\
 E_{12}&=&R^{3}E_{11}R^{-3}=H_{b-2,b}B_{k-1}A_{k+1}B_{k+2}A_{k+4}B_{k+7} \textrm { and }\\
 E_{13}&=&(E_{11}E_{12})E_{11}(E_{11}E_{12})^{-1}=H_{b+1,b+3}B_{k-4}A_{k-2}B_{k-1}A_{k+1}A_{k+4},
 \end{eqnarray*}
 which are contained in $H$. Hence, $H$ contains the element 
$E_{13}E_{11}^{-1}=A_{k+4}B_{k+4}^{-1}$. Hence, $A_{i}B_{i}^{-1}\in H$ for $i=1,\ldots g$,
 by conjugating $A_{k+4}B_{k+4}^{-1}$ with powers of $R$.
Finally, we obtain the following elements:
\begin{eqnarray*}
A_1A_{2}^{-1}&=&(A_1B_1^{-1})(B_1C_1^{-1})(C_1B_{2}^{-1})(B_{2}A_{2}^{-1}),\\
B_1B_{2}^{-1}&=&(B_1C_1^{-1})(C_1B_{2}^{-1}) \textrm{ and }\\
C_1C_{2}^{-1}&=&(C_1B_{2}^{-1})(B_{2}C_{2}^{-1}),
\end{eqnarray*}
which are all contained in $H$. This completes the proof for $p=2b+1\geq 7$ by Theorem~\ref{thm1}.

For $p=2b\geq 6$, one can replace $H_{b,b+2}$ with $H_{b,b+1}$ and follow exactly the same steps as above.
\end{proof}


\begin{lemma}\label{lem2k+1}
For $g=2k+1\geq 13$, the subgroup $H$ of $\mod^{*}(\Sigma_{g,p})$ generated by 
\begin{equation*}
 \begin{cases}
      \rho_1, \rho_2, \rho_2H_{b,b+2}A_{k-1}C_{k-3}B_{k+1}C_{k+4}A_{k+3} & \text {if  $p=2b+1\geq 7,$}\\
        \rho_1, \rho_2, \rho_2H_{b,b+1}A_{k-1}C_{k-3}B_{k+1}C_{k+4}A_{k+3}&  \text{if $p=2b\geq 6$}
    \end{cases}       
\end{equation*}
contains the Dehn twists $A_i$, $B_i$ and $C_j$ for $i=1,\ldots,g$ and $j=1,\ldots,g-1$.
\end{lemma}
\begin{proof}
Consider the models for $\Sigma_{g,p}$ as shown in Figures~\ref{OO} and~\ref{OE}. First let us consider the case $p=2b+1$. Let $F_1=H_{b,b+2}A_{k-1}C_{k-3}B_{k+1}C_{k+4}A_{k+3}$ so that the subgroup $H$ generated by the elements $\rho_1$, $\rho_2$ and $\rho_2F_1$. It follows from $H$ contains the elements $\rho_1$, $\rho_2$ and $\rho_2F_1$ that $H$ also contains the elements $R=\rho_1\rho_2$ and $F_1=\rho_2\rho_2F_1$.

Let $F_2$ denote the conjugation of $F_1$ by $R^{-1}$ so that 
\[
F_2=R^{-1}F_1R=H_{b+1,b+3}A_{k-2}C_{k-4}B_{k}C_{k+3}A_{k+2}\in H.
\]
and let $F_3$ be the conjugation of $F_2$ by $R^{3}$: 
\[
F_3=R^{3}F_2R^{-3}=H_{b-2,b}A_{k+1}C_{k-1}B_{k+3}C_{k+6}A_{k+5}\in H.
\]
From these, we get the following element:
\begin{eqnarray*}
F_4&=&(F_2F_3)F_2(F_2F_3)^{-1}\\
&=&H_{b+1,b+3}A_{k-2}C_{k-4}C_{k-1}B_{k+3}A_{k+2},
\end{eqnarray*}
which is contained in $H$. Thus, the subgroup $H$ contains the element 
\[
F_5=F_4F_2^{-1}=C_{k-1}B_{k+3}C_{k+3}^{-1}B_{k}^{-1}.
\]
Also we get the following elements:
\begin{eqnarray*}
F_6&=&R^{3}F_4R^{-3}=H_{b-2,b}A_{k+1}C_{k-1}C_{k+2}B_{k+6}A_{k+5} \textrm{ and }\\
F_7&=&(F_4F_6)F_4(F_4F_6)^{-1}=H_{b+1,b+3}A_{k-2}C_{k-4}C_{k-1}C_{k+2}A_{k+2},
\end{eqnarray*}
which are contained in $H$.
Hence, we see that the element $F_7F_{4}^{-1}=C_{k+2}B_{k+3}^{-1} \in H$, which implies that $C_{i}B_{i+1}^{-1} \in H$ for all $i=1,\ldots,g-1$ by the action of $R$. It follows from the element $B_{k}C_{k-1}^{-1}\in H$ that
$F_5(B_{k}C_{k-1}^{-1})=B_{k+3}C_{k+3}^{-1}
$ is also contained in $H$. Similarly we have $B_{i}C_i^{-1} \in H$ for all $i=1,\ldots,g-1$ by the action of $R$. Moreover, the elements
 \begin{eqnarray*}
 F_8&=&(B_{k+1}C_{k}^{-1})(C_{k}B_{k}^{-1})F_2\\
 &=&H_{b+1,b+3}A_{k-2}C_{k-4}B_{k+1}C_{k+3}A_{k+2}, \\
 F_9&=&R^{3}F_8R^{-3}=H_{b-2,b}A_{k+1}C_{k-1}B_{k+4}C_{k+6}A_{k+5} \textrm { and }\\
 F_{10}&=&(F_8F_9)F_8(F_8F_9)^{-1}=H_{b+1,b+3}A_{k-2}C_{k-4}A_{k+1}B_{k+4}A_{k+2}
 \end{eqnarray*}
are all in $H$. Thus $H$ contains the element 
$F_8F_{10}^{-1}(B_{k+4}C_{k+3}^{-1})=B_{k+1}A_{k+1}^{-1}$. Hence, $B_{i}A_{i}^{-1}\in H$ for $i=1,\ldots, g$ by conjugating this element with powers of $R$.
The remaining part of the proof can be completed as in the proof of Lemma~\ref{lem2k}.
\end{proof}
\begin{lemma}\label{lem11}
For $g=11$, the subgroup $H$ of $\mod^{*}(\Sigma_{g,p})$ generated by 
\begin{equation*}
 \begin{cases}
      \rho_1, \rho_2, \rho_1H_{b,b+1}B_1A_4C_6A_9 & \text {if  $p=2b+1\geq 15,$}\\
        \rho_1, \rho_2, \rho_1H_{b-1,b+1}B_1A_4C_6A_9 &  \text{if $p=2b\geq 16$}
    \end{cases}       
\end{equation*}
contains the Dehn twists $A_i$, $B_i$ and $C_j$ for $i=1,\ldots,g$ and $j=1,\ldots,g-1$.
\end{lemma}
\begin{proof}
Consider the models for $\Sigma_{g,p}$ as shown in Figures~\ref{OO} and~\ref{OE}. Let us first consider the case $p=2b+1$. Let $G_1=H_{b,b+1}B_1A_4C_6A_9$ and $H$ be the group generated by the elements $\rho_1$, $\rho_2$ and $\rho_1G_1$. It is easy to see that $H$ contains the elements $R=\rho_1\rho_2$ and $G_1=\rho_1\rho_1G_1$. We then have the following elements:
\begin{eqnarray*}
G_2&=&R^{-3}G_1R^3=H_{b+3,b+4}B_9A_1C_3A_6, \\
G_3&=&(G_1G_2)G_1(G_1G_2)^{-1}=H_{b,b+1}A_1A_4C_6B_9,\\
G_4&=&R^{3}G_3R^{-3}=H_{b-3,b-2}A_4A_7C_9B_1,\\
G_5&=&(G_4G_3)G_4(G_4G_3)^{-1}=H_{b-3,b-2}A_4A_7B_9A_1,\\
G_6&=&R^3G_5R^{-3}=H_{b-6,b-5}A_7A_{10}B_1A_4\textrm{ and }\\
G_7&=&(G_5G_6)G_5(G_5G_6)^{-1}=H_{b-3,b-2}A_4A_7B_9B_1,
\end{eqnarray*}
which are all in $H$. Thus, we obtain the element $G_5G_7^{-1}=A_1B_1^{-1}$. By conjugating by powers of $R$, we see that $A_iB_i^{-1}\in H$ for $i=1,2,\ldots,g$. We also have 
\begin{eqnarray*}
G_8&=&(B_4A_4^{-1})G_1=H_{b,b+1}B_1B_4C_6A_9 \in H,\\
G_{9}&=&R^{-3}G_8R^3=H_{b+3,b+4}B_9B_1C_3A_6 \in H \textrm{ and }\\
G_{10}&=&(G_{9}G_8)G_{9}(G_{9}G_8)^{-1}=H_{b+3,b+4}A_9B_1B_4A_6\in H.
\end{eqnarray*}
Hence, we get $G_{9}G_{10}^{-1}(A_9B_9^{-1})=C_3B_4^{-1}\in H$, which implies that $C_iB_{i+1}\in H$ for $i=1,2,\ldots,g-1$ by the action of $R$. Moreover, the subgroup $H$ contains the following elements:
\begin{eqnarray*}
G_{11}&=&(B_9A_9^{-1})G_1=H_{b,b+1}B_1A_4C_6B_9,\\
G_{12}&=&R^{-3}G_{11}R^3=H_{b+3,b+4}B_9A_1C_3B_6  \textrm{ and } \\
G_{13}&=&(G_{12}G_{11})G_{12}(G_{12}G_{11})^{-1}=H_{b+3,b+4}B_9B_1C_3C_6.
\end{eqnarray*}
It follows that $G_{12}G_{13}^{-1}(B_1A_1^{-1})=B_6C_6^{-1}$. Again, by the action of $R$, the elements $B_iC_i^{-1}\in H$. One can complete the remaining part of the proof as in the proof of Lemma~\ref{lem2k}.
\end{proof}

\begin{figure}[hbt!]
\begin{center}
\scalebox{0.25}{\includegraphics{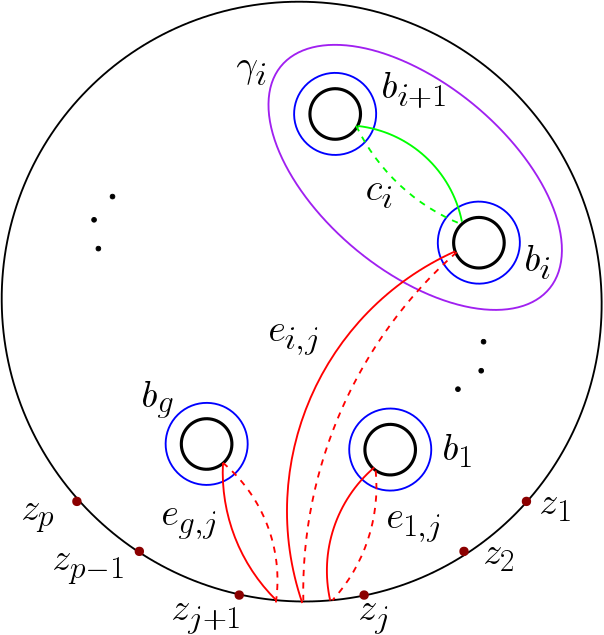}}
\caption{The curves $\gamma_i$ and $e_{i,j}$ on the surface $\Sigma_{g,p}$.}\label{C}
\end{center}
\end{figure}
\

\begin{lemma}\label{lemma3}
Let $g\geq2$. For $i=1,\ldots,g-1$, in the mapping class group $\mod(\Sigma_{g,p})$, the element 
\[
\phi_i=B_{i+1}\Gamma_{i}^{-1} C_{i}B_{i}
\]
maps the curve $e_{i,j}$ to the curve $e_{i+1, j}$, where the curves $\gamma_i$ and $e_{i,j}$'s are as in Figure~\ref{C}. Moreover, the diffeomorphism $\phi_i$ is contained in the group $H$ for $i=1,\ldots,g-1$.
\end{lemma}
\begin{proof}
It is easy to see that the diffeomorphism $\phi_i$ maps $e_{i,j}$ to $e_{i+1,j}$. Consider the diffeomorphism 
\[
S=A_1B_1C_1\cdots C_{g-2}B_{g-1}C_{g-1}B_g. 
\]
Since $S \in H$ and $S$ maps $a_2$ to $ \gamma_1$, the element $SA_2S^{-1}=\Gamma_1\in H$. By conjugating with powers of $R$, the element $\Gamma_i$ is in $H$. We conclude that $\phi_i\in H$.
 \end{proof}

Let $H$ be the subgroup of $\mod^{*}(\Sigma_{g,p})$
generated by the elements given explicitly in lemmata~\ref{lem2k}, ~\ref{lem2k+1} and~\ref{lem11}  with the conditions mentioned in these lemmata.

\begin{lemma}\label{lemma4}
The group $\mod_{0}^{*}(\Sigma_{g,p})$ is contained in the group $H$.
\end{lemma}
\begin{proof}
Since the group $H$ contains the Dehn twists $A_1$, $A_2$, $B_1,B_2,\ldots,B_g$ and $C_1,C_2,\ldots,C_{g-1}$ by lemmata~\ref{lem2k}, ~\ref{lem2k+1} and~\ref{lem11}, it suffices to prove that $H$ also contains the elements $E_{i.j}$ for some fixed $i$ and $j=1,2,\ldots,p-1$. First note that $H$ contains $A_{g}$ and $R=\rho_1\rho_2$. Consider the models for $\Sigma_{g,p}$ as shown in Figures~\ref{EO},~\ref{EE}.~\ref{OO} and~\ref{OE},  Since the diffeomorphism $R$ maps $a_{g}$ to $e_{1,p-1}$, we have 
\[
RA_{g}R^{-1}=E_{1,p-1} \in H.
\]
The diffeomorphism $\phi_{g-1}\cdots \phi_2\phi_1$ in Lemma~\ref{lemma3} is given by 
$\phi_i=B_{i+1}\Gamma_i^{-1}C_iB_i$ which maps each $e_{i,j}$ to $e_{i+1,j}$ for $j=1,2,\ldots,p-1$ (see Figure~\ref{C}). So we get
\[
\phi_{g-1}\cdots \phi_2\phi_1E_{1,p-1}(\phi_{g-1}\cdots \phi_2\phi_1)^{-1}=E_{g,p-1}\in H.
\]
Similarly, the diffeomorphism $R$ sends $e_{g,p-1}$ to $e_{1,p-2}$. Then we have
\[
RE_{g,p-1}R^{-1}=E_{1,p-2}\in H.
\]
It follows from 
\[
\phi_{g-1}\cdots \phi_2\phi_1E_{1,p-2}(\phi_{g-1}\cdots \phi_2\phi_1)^{-1}=E_{g,p-2}\in H
\]
 that
 \[
 R(E_{g,p-2})R^{-1}=E_{1,p-3}\in H.
 \]
 Continuing in this way, we conclude that the elements $E_{1,1},E_{1,2},$ $\ldots,E_{1,p-1}$ are contained in $H$. This completes the proof.
\end{proof}
We thank the referee for pointing us the proof of the following lemma.
\begin{lemma}\label{symm}
The symmetric group $S_{2b+1}$ is generated by the transposition $(b,b+2)$ and the $(2b+1)$-cycle $(1,2,\ldots,2b+1)$.
\end{lemma}
\begin{proof}
Set $\tau=(b,b+2)$ and $\sigma=(1,2,\ldots,2b+1)$. It is easy to verify that
\[
\sigma^{2}=(1,3,5,\ldots,2b+1,2,4,6,\ldots,2b).
\]
Now, rewrite $s_i=2i-1$ for $i=1,2,\ldots, b+1$ and $s_{b+1+i}=2i$ for $i=1,2,\ldots, b$. This gives
\[
\sigma^{-b+1}\tau\sigma^{b-1}=(s_1,s_2),
\]
\[
\sigma^{2}=(s_1,s_2,\ldots,s_{2b+1}).
\]
Since $(s_1,s_2)$ and $(s_1,s_2,\ldots,s_{2b+1})$ generate $S_{2b+1}$, we see that $S_{2b+1}=\langle \tau,\sigma \rangle$.
\end{proof}

Now, we are ready to prove the main theorem of this section.

\textit{Proof of Theorem B.}
 Consider  the surface $\Sigma_{g,p}$ as in Figures~\ref{EO}
and~\ref{EE}.\\
\underline{ If $g=2k\geq 10$ and $p\geq 6$}: In this case, consider the surface $\Sigma_{g,p}$ as in Figures~\ref{EO}
and~\ref{EE}. Since
\[
\rho_2(b_{k-3})=b_{k+4}, \rho_2(a_{k-1})=a_{k+2} \textrm{ and }\rho_2(c_{k})=c_{k}
\]
and $\rho_2$ is an orientation reversing diffeomorphism, we get
\[
\rho_2B_{k-3}\rho_2=B_{k+4}^{-1},
\rho_2A_{k-1}\rho_2=A_{k+2}^{-1} \textrm{ and }
\rho_2C_{k}\rho_2=C_{k}^{-1}.
\]
Also, observe that $\rho_2H_{b,b+2}\rho_2=H_{b,b+2}^{-1}$ for $p=2b+1$ and $\rho_2H_{b,b+1}\rho_2=H_{b,b+1}^{-1}$ for $p=2b$.
Then it is easy to see that each 
  \begin{equation*}
 \begin{cases}
       \rho_2H_{b,b+2}B_{k-3}A_{k-1}C_kA_{k+2}B_{k+4} & \text {if  $p=2b+1,$}\\
        \rho_2H_{b,b+1}B_{k-3}A_{k-1}C_kA_{k+2}B_{k+4} &  \text{if $p=2b$}
    \end{cases}       
\end{equation*}   
is an involution. Therefore, the generators of the subgroup $H$ given in Lemma~\ref{lem2k} are involutions.

\underline{ If $g=2k+1\geq 13$ and $p\geq 6$}: In this case, consider the surface $\Sigma_{g,p}$ as in Figures~\ref{OO}
and~\ref{OE}. It follows from
\[
\rho_2(a_{k-1})=a_{k+3}, \rho_2(c_{k-3})=c_{k+4} \textrm{ and }\rho_2(b_{k+1})=b_{k+1}
\]
and $\rho_2$ is an orientation reversing diffeomorphism that
\[
\rho_2A_{k-1}\rho_2=A_{k+3}^{-1},
\rho_2C_{k-3}\rho_2=C_{k+4}^{-1} \textrm{ and }
\rho_2B_{k+1}\rho_2=B_{k+1}^{-1}.
\]
Also, by the fact that $\rho_2H_{b,b+2}\rho_2=H_{b,b+2}^{-1}$ for $p=2b+1$ and $\rho_2H_{b,b+1}\rho_2=H_{b,b+1}^{-1}$ for $p=2b$,
it is easy to see that the elements
\begin{equation*}
 \begin{cases}
  \rho_2H_{b,b+2}A_{k-1}C_{k-3}B_{k+1}C_{k+4}A_{k+3} & \text {if  $p=2b+1$,}\\
       \rho_2H_{b,b+1}A_{k-1}C_{k-3}B_{k+1}C_{k+4}A_{k+3}&  \text{if $p=2b$}
    \end{cases}       
\end{equation*}
are involutions.

\underline{ If $g=11$ and $p\geq 15$}: Consider the surface $\Sigma_{g,p}$ as in Figures~\ref{OO}
and~\ref{OE}. It is easy to see that
\[
\rho_1(b_{1})=b_{1}, \rho_1(a_{4})=a_{9} \textrm{ and }\rho_1(c_{6})=c_{6}
\]
and $\rho_1$ is an orientation reversing diffeomorphism that
\[
\rho_1B_{1}\rho_1=B_{1}^{-1},
\rho_1A_{4}\rho_1=A_{9}^{-1} \textrm{ and }
\rho_1C_{6}\rho_1=C_{6}^{-1}.
\]
Also, since $\rho_1H_{b,b+1}\rho_1=H_{b,b+1}^{-1}$ for $p=2b+1$ and $\rho_1H_{b-1,b+1}\rho_1=H_{b-1,b+1}^{-1}$ for $p=2b$, it is easy to verify that the elements
\begin{equation*}
 \begin{cases}
  \rho_1H_{b,b+1}B_{1}A_{4}C_{6}A_{9} & \text {if  $p=2b+1$,}\\
       \rho_1H_{b-1,b+1}B_{1}A_{4}C_{6}A_{9} &  \text{if $p=2b$}
    \end{cases}       
\end{equation*}
are involutions. We see that the generators of the subgroup $H$ given in Lemma~\ref{lem11} are involutions.

The group $\mod_{0}^{*}(\Sigma_{g,p})$ is contained in $H$ by Lemma~\ref{lemma4}. We finish the proof by showing that $H$ is mapped surjectively onto $S_p$ by Lemma~\ref{lemma1}: The subgroup $H$ contains the element $\rho_2\rho_1$ which has the image $(1,2,\ldots,p)\in S_p$. For $g\neq 11$, since the subgroup $H$ contains the Dehn twists $A_i$, $B_i$ and $C_i$ by lemmata~\ref{lem2k} and~\ref{lem2k+1} , the group $H$ contains the half twist $H_{b,b+2}$ if $p=2b+1$ and the half twist $H_{b,b+1}$ if $p=2b$. For $p=2b+1$, it follows from Lemma~\ref{symm} that the image of $H_{b,b+2}$ which is $(b,b+2)$ and the $p$-cycle $(1,2,\ldots,p)$ generate $S_p$. For $p=2b$, it is clear that the image of $H_{b,b+1}$ which is $(b,b+1)$ and again the $p$-cycle $(1,2,\ldots,p)$ generate $S_p$. Likewise, for $g=11$, by Lemma~\ref{lem11}, the subgroup $H$ contains the half twist $H_{b,b+1}$ if $p=2b+1$, the half twist $H_{b-1,b+1}$ if $p=2b$. For the latter case $H$ also contains the half twist $R^{-1}H_{b-1,b+1}R=H_{b,b+2}$. This finishes the proof by the above argument.

Before we finish the paper let us mention the cases $p=2$ or $p=3$.
In these cases, the generating set of $H$ can be chosen as 
\[
H=\left\{\begin{array}{lll}
\lbrace \rho_1,\rho_2, \rho_2B_{k-3}A_{k-1}C_kA_{k+2}B_{k+4} \rbrace & \textrm{if} & g=2k\geq10,\\
\lbrace \rho_1,\rho_2,\rho_2A_{k-1}C_{k-3}B_{k+1}C_{k+4}A_{k+3} \rbrace & \textrm{if} & g=2k+1\geq13.\\
\lbrace \rho_1,\rho_2,\rho_1B_1A_4C_6A_9\rbrace & \textrm{if} & g=11.\\
\end{array}\right.
\]
One can easily prove that the group $H$ contains $\mod_{0}^{*}(\Sigma_{g,p})$ by the similar arguments in the proofs of  lemmata~\ref{lem2k+1}, \ref{lem2k}, \ref{lem11} and \ref{lemma4}. The element $\rho_2\rho_1 \in H$ has the image $(1,2,\ldots,p)\in S_p$. Thus, for $p=2$ this element generates $S_p$. If $p=3$,  the element $\rho_1$ has the image $(1,2)$. Therefore, the group $H$ is mapped surjectively onto $S_p$ for $p=2,3$, We conclude that the group $H$ is equal to $\mod^{*}(\Sigma_{g,p})$.


\end{document}